\documentclass[12pt,a4paper]{article}

\usepackage{amsmath}
\usepackage{amsthm}
\usepackage{amssymb}
\usepackage{amsfonts}
\usepackage{tikz}
\usepackage{geometry}
\usepackage{enumitem}
\usepackage[numbers]{natbib}
\usetikzlibrary{calc}
\usetikzlibrary{decorations.markings}

\tikzset{->-/.style={decoration={
			markings,
			mark=at position .6 with {\arrow{>}}},postaction={decorate}}}

\newtheorem{Satz}{Proposition}[section]
\newtheorem{Proposition}[Satz]{Proposition}
\newtheorem{Korollar}[Satz]{Corollary}
\newtheorem{Lemma}[Satz]{Lemma}
\newtheorem{Theorem}[Satz]{Theorem}

\newtheorem{Theorem*}[]{Theorem}

\theoremstyle{definition}

\theoremstyle{remark}

\newtheorem{Bemerkung}[Satz]{Remark}
\newtheorem{Beispiel}[Satz]{Example}

\DeclareMathOperator{\prodG}{prod}

\newcommand{\Addresses}{{
		\bigskip
		\footnotesize
		
		Mireille Soergel, \textsc{IMB, UMR 5584, CNRS, Univ.  Bourgogne Franche-Comt\'e, 21000 Dijon, France}\par\nopagebreak
		\textit{E-mail address}: \texttt{mireille.soergel@u-bourgogne.fr}

}}
\title{Systolic complexes and group presentations}
\author{Mireille Soergel}
\date{}
\begin{document}
 \maketitle
 \begin{abstract}
 	We give conditions on a presentation of a group, which imply that its Cayley complex is simplicial and the flag complex of the Cayley complex is systolic. We then apply this to Garside groups and Artin groups. We give a classification of the Garside groups whose presentation using the simple elements as generators  satisfy our conditions. We then also give a dual presentation for Artin groups and identify in which cases the flag complex of the Cayley complex is systolic.
 \end{abstract}
 
 \section{Introduction}

In order to better understand a group, one approach is to understand whether it acts properly discontinuously and cocompactly by isometries on a non-positively curved space. For Riemannian manifolds, one could consider the sectional curvature of the manifold. More generally for geodesic metric spaces, we see if they satisfy the CAT(0) inequality. For simplicial complexes Januszkiewicz and Swiatkowski introduced in \cite{JanSwi} the notion of systolic complexes as a combinatorial form of non positive curvature. Though their 1-skeleta had been studied earlier by Chepoi under the name bridged graphs \cite{Chepoi}. A flag simplicial complex is systolic if it is simply connected and all of its vertex links are 6-large, i.e. all cycles of length 4 or 5 have diagonals. The main consequences for a group $G$ acting properly discontinuously and cocompactly on a systolic complex are the following:



\begin{enumerate}
 \item $G$ is biautomatic. This was already proven by Januszkiewicz and Swiatkowski in their original paper on systolicity \cite{JanSwi}. This especially implies solvability of the Word Problem and of the Conjugacy Problem and a quadratic Dehn function. 
 \item Every finitely presented subgroup of $G$ is systolic. Wise showed this for finitely presented subgroups of torsion-free subgroups. For all systolic groups, this has been shown in \cite{Zadnik},\cite{HanMar}.
 \item Virtually solvable subgroups of $G$ are either virtually cyclic or virtually $\mathbb{Z}^2$. This is mentioned in \cite{HuaOsa}. 
 Bestvina showed that virtually solvable subgroups of spherical Artin groups are abelian \cite{Bestvina}. Gersten-Short showed a result for polycyclic subgroups of biautomatic groups \cite{GerSho}.
 \item The centralizer of an infinite-order element of $G$ is commensurable with \linebreak $F_n\times\mathbb{Z}$ or $\mathbb{Z}$ \cite{Crisp}, \cite{Elsner}.
\end{enumerate}
 As a group naturally acts on its Cayley graph, this leads to the following questions: How do we construct a flag simplicial complex from the Cayley graph? Can we give conditions on the presentation which ensure that this space is systolic? The goal of this paper is to give a partial answer to these questions. 
Systolicity is already known for large type Artin groups \cite{HuaOsa} and right angled Artin groups with bipartite defining graph \cite{ElsPrz}. Triangular Coxeter groups (except (2,4,4),(2,4,5),(2,5,5)) are also systolic \cite{PrzSch}. In this paper we will focus our applications on Garside groups and Artin groups.
We will first give some background on systolic complexes and set some conditions on group presentation that lead to a simplicial Cayley complex with free group action. We introduce the notion of a restricted triangular presentation. We say that a presentation $\langle S\mid R\rangle$ of some groupe $G$ is a restricted triangular presentation if $S\cap S^{-1}=\emptyset$, $R = \{a\cdot b\cdot c^{-1}\mid abc^{-1} =e \text{ in } G\}$ and for $a,b,c\in S$ $abc \in S \Rightarrow ab, bc\in S$. In this case the Cayley complex is simplicial. We then establish when the flag complex of such a Cayley complex is sytolic. To do so we study the cycles of length four and five. This leads to

\begin{Theorem*}(Theorem \ref{systolic})\label{1}
	Consider a group $G$ with generating set $S$, where $G$ has a finite restricted triangular presentation with respect to $S$. Then the complex $Flag(G,S)$ is a simply connected simplicial complex. It is systolic if and only if the generating set $S$ satisfies the following conditions: 
	\begin{enumerate}[label=\arabic*)]
		
		\item If $\exists u, w, a, b, c, d\in S$, $u\neq w$, $a\neq d$, with $ua=wb\in S$ and $ud=wc\in S$ then $\exists k\in S$ such that $w=uk$ or $ua = udk$.
		\item If $\exists v,x, a, b, c, d\in S$, $v\neq x$, $a\neq b$, with $bv = cx\in S$ and $av=dx\in S$ then $\exists k\in S$ such that $v= kx$ or $ av = kbv$.
		
		\item If $\exists u, v, x, b, c\in S$, $v\neq x$,  with $ux\in S$, $uv \in S$ and $vb= xc\in S$ then $\exists k\in S$ such that $k = uvb$ or $v = xk$.

		\item If $\exists v, w, x, a, d\in S$, $v\neq x$, with $vw\in S$, $xw\in S$ and $dx= av\in S$ then $\exists k\in S$ such that $k = avw$ or $x=kv$.
	 
		\item If $\exists u, v, w, x\in S$, $v\neq x$, $u\neq w$, with $wv\in S$, $wx\in S$, $uv\in S$ and $ux\in S$ then $\exists k\in S$ such that $w=ku$ or $x = vk$.
		
	\end{enumerate}

	Moreover, this implies that $G$ is a systolic group.
\end{Theorem*}

We say that a presentation satisfying the conditions of Theorem 1 is a systolic presentation.
  We present some examples which show that these conditions are all necessary.

   There are two main applications of Theorem \ref{1}. The first concerns Garside groups. They were introduced by Dehornoy and Paris in \cite{DehPar} as a generalization of spherical Artin groups. The Garside structure on a group naturally gives a presentation leading to a simplicial Cayley graph, we call this the Garside presentation of a Garside group. We can classify the Garside groups whose Garside presentation satisfy Theorem \ref{systolic}. Consider the following definitions. Let $x_1,\dots, x_n$ be $n$ letters and let $m$ be a positive integer. We define $$\prodG(x_1,\dots, x_p;m) = \underbrace{x_1x_2\dots x_px_1x_2\dots}_{m}.$$ and $\prodG(x_1,\dots,x_p;0) = e$. Consider the group \begin{equation*}\begin{split}G_{n,m} = \langle x_1,\dots, x_n\mid &\prodG(x_1,\dots,x_n;m) = \prodG(x_2,\dots,x_{n}, x_1;m)=\dots\\ &=\prodG(x_n, x_1,\dots,x_{n-1};m)\rangle.\end{split}\end{equation*} We can then make the following statement about systolic Garside groups:
   
   \begin{Theorem*}(Theorem \ref{systolic Garside}) Let $G$ be a Garside group of finite type.  Then $G$ has a systolic Garside presentation if and only if $G = (\ast_{i=1}^p G_{n_i,m_i})/(\Delta_{n_i,m_i}=\Delta_{n_j,m_j} \forall i, j)$ for some positive integers $p, n_1,\dots, n_p$ and $m_1,\dots, m_p$.
   	
   \end{Theorem*}

The second application concerns Artin groups. To state the next result we require a few definitions. An \textit{orientation} on a simple graph $\Gamma$ is an assignement $o(e)$ for each edge $e\in E(\Gamma)$ where $o(e)$ is a set of one or two endpoints of $e$. An edge with both endpoints assigned is \textit{bioriented}. The startpoint $i(e)$ is an assignement of one or two startpoints of $e$, which is consistent with the choice of $o(e)$. If $o(e)$ consists of one point, $i(e)$ consists of one point such that $e = (i(e),o(e))$, if $o(e)$ consists of two points then so does $i(e)$.
We say that a cycle $\gamma$ is \textit{directed} if for each $v\in\gamma$ there is exactly one edge $e\in\gamma$ with $v\in o(e)$. A cycle is \textit{undirected} if it is not directed.
We say that a 4-cycle $\gamma = (a_1,a_2,a_3,a_4)$ is \textit{misdirected} if 
$a_2\in o(a_1,a_2), a_2\in o(a_2,a_3), a_4\in o(a_3,a_4)$ and $a_4\in o(a_4,a_1)$.

\begin{tikzpicture}

\draw[fill] (0,0) circle [radius=0.05];
\draw[fill] (1,0) circle [radius=0.05];
\draw[fill] (0,1) circle [radius=0.05];
\draw[fill] (1,1) circle [radius=0.05];
\draw[->-, >=latex](0,0)node[left]{$a_1$}--(1,0)node[right]{$a_4$};
\draw[->-, >=latex](1,1)node[right]{$a_3$}--(1,0);
\draw[->-, >=latex](0,0)--(0,1)node[left]{$a_2$};
\draw[->-, >=latex] (1,1)--(0,1);
\end{tikzpicture}

Consider the notation $\left[xyx\dots\right]_k = \underbrace{xyx\dots}_{k}\ $ and $\left[\dots xyx\right]_k = \underbrace{\dots xyx}_{k}\ $ for some $k\in\mathbb{N}$. Given a finite labeled simple graph $\Gamma$, the \textit{Artin group associated to $\Gamma$} is given by \begin{multline*}A_{\Gamma} = \langle s_v, v\in V\mid \left[s_vs_ws_v\dots\right]_{m_e} = \left[s_ws_vs_w\dots\right]_{m_e} \\ \text{ for all edges } e = (v,w) \text{ with label } m_e\rangle.\end{multline*}
 We can now present the following result:

\begin{Theorem*}(Theorem \ref{Artin})
	Let $\Gamma$ be a simple graph, with edges labeled by numbers $\geq 2$ and with an orientation $o$ such that an edge is bioriented if and only if it has label $2$. Assume that 
	every 3-cycle is directed and no 4-cycle is misdirected. Let $A$ be the Artin group associated to $\Gamma$.  Then the dual presentation of $A_{\Gamma}$ induced by this orientation is systolic.
\end{Theorem*}

\textbf{Acknowledgements} The author thanks their advisors Thomas Haettel and Luis Paris for many helpful discussions and advice. 
\newline
The author is partially supported by the French project “AlMaRe” (ANR-19-CE40-0001-01) of the ANR.
   

 \section{Preliminaries}


We start with some background on simplicial and systolic complexes. Let $X$ be a simplicial complex. Assume it is finite dimensional and locally finite. We denote its $k$-skeleton by $X^{(k)}$. Then $X^{(0)}$ is the set of vertices of $X$. The subcomplex spanned by $A\subset X^{(0)}$ is the largest subcomplex of $X$ which has $A$ as its set of vertices.
The complex $X$ is \textit{flag} if every set of pairwise adjacent vertices spans a simplex. A flag complex is uniquely determined by its 1-skeleton. For a simplex $\sigma\in X$ we can define its \textit{link} in $X$,  $$\text{Lk}(\sigma, X) = \{\tau\in X\mid \tau\cap\sigma = \emptyset \\ \text{ and } \tau\cup\sigma\in X\}.$$
A \textit{cycle} in X is the image of a simplicial map $f: S^1\rightarrow X$ from a triangulation of the 1-sphere to $X$. If $f$ is injective, the cycle is embedded. Let $\gamma$ be an embedded cycle in $X$. The \textit{length} of $\gamma$, $|\gamma|$, is the number of edges of $\gamma$. We say that $\gamma$ is a \linebreak \textit{$|\gamma|$-cycle}. A \textit{diagonal} of $\gamma$ is an edge that connects two nonconsecutive vertices of $\gamma$. An embedded cycle is \textit{diagonal free} if there are no edges between nonconsecutive vertices. We say that two vertices $v$ and $w$ are adjacent if there exists an edge between $v$ and $w$, we then write $v \sim w$.
A simplicial complex is \textit{6-large} if every embedded cycle $\gamma$ with $4\leq|\gamma|<6$ has a diagonal.
A simplicial complex is \textit{systolic} if it is connected, simply connected and if Lk$(v, X)$ is flag and 6-large for all vertices $v\in X$. A group is \textit{systolic} if it acts properly discontinuously and cocompactly on a systolic complex. To know when a Cayley complex is systolic we first need to determine when it is simplicial.

Let $G$ be a group and $S\subset G$ a finite set of generators. Suppose additionally that $S\cap S^{-1} = \emptyset$, this especially implies $e\notin S$ and $s^2 \neq e$ for all $s\in S$. Let $\Gamma(G,S)$ be the Cayley graph of $G$ relative to $S$. Its vertices and edges are \linebreak $V(\Gamma(G,S)) =  \{v[g]\mid g\in G\}$ and $E(\Gamma(G,S)) = \{e[g,s]\mid g\in G, s\in S\}$ where the edge $e[g,s]$ goes from $v[g]$ to $v[gs]$. We also write $e[g,s] = (v[g],v[gs])$. As $S\cap S^{-1} = \emptyset$, the graph $\Gamma(G,S)$ is simplicial. So we can define $Flag(G,S)$ as the flag complex of $\Gamma(G,S)$. As $Flag(G,S)$ is the flag complex of $\Gamma(G,S)$, the group $G$ naturally acts properly discontinuously and cocompactly by isometries on $Flag(G,S)$.

\begin{Proposition}\label{triangular} Let $G$ be a group and $S \subset G$ a finite generating set. Suppose additionally that $S\cap S^{-1} = \emptyset$. Then $Flag(G,S)$ is a simply connected simplicial complex and $\pi_1(Flag(G,S)/G) = G$ if and only if $G$ admits the presentation $G = \langle S\mid R\rangle $ where
 $ R = \{a\cdot b\cdot c\mid a, b ,c\in S \text{ with }abc= e \text{ in } G\}\cup \{a\cdot b\cdot c^{-1}\mid a, b ,c\in S \text{ with } abc^{-1} =e \text{ in } G\}$.
\end{Proposition}

\begin{proof}

To see when $Flag(G,S)$ is simply connected it is enough to take a look at its 2-skeleton. There is a 2-simplex in $Flag(G,S)$ for every set of 3 pairwise adjacent vertices. As in $\Gamma(G,S)$ edges are labeled by elements in $S$ and vertices correspond to elements of $G$, so are edges and vertices in $Flag(G,S)$. Hence we can interpret the existence of a 2-simplex in terms of relations on the generators. At each vertex $v[g]$ in $Flag(G,S)$ and for every triple $a, b, c\in S$ with $a\cdot b\cdot c = e$ or $a\cdot b\cdot c^{-1} = e$ in $G$ there is a 2-simplex with vertices $v[g]$, $v[ga]$, $v[gab]$ and edges $e[g,a]$, $e[ga,b]$, $e[gab,c]$ or $e[g,a]$, $e[ga,b]$, $e[g,c]$. On the other hand each 2-simplex has vertices and edges which correspond to such a triple of generators. 
This implies that \begin{multline*}
\pi_1(Flag(G,S)/G) = \langle S\mid a\cdot b\cdot c = e \text{ or } a\cdot b\cdot c^{-1} = e \text{ for all } a, b, c \\ \text{for which one of these equalities holds in $G$}\rangle. \end{multline*} Finally $Flag(G,S)$ is simply connected if and only if it is the universal cover of the quotient, so if and only if $\pi_1(Flag(G,S)/G) = G$.
\end{proof}

We call a presentation satisfying the conditions of Proposition \ref{triangular} a \textit{triangular} presentation. 

\begin{Proposition}\label{free}
 Assume a group $G$ has a triangular presentation $\langle S\mid R\rangle$. Assume additionally $R = \{a\cdot b\cdot c^{-1}\mid a, b ,c\in S \text{ with } abc^{-1} =e \text{ in } G\}$. Then the action of $G$ on $Flag(G,S)$ is free.
\end{Proposition}
\begin{proof} 
 Let $g\in G$ and $x\in Flag(G,S)$ such that $g\cdot x= x$. Let $V$ be the set of vertices of the smallest simplex containing $x$. As the action of $G$ is simplicial, $g\cdot V = V$. 
 The restriction of possible relations in $R$ imposes an orientation on triangles in the graph, which in turn implies that in the subcomplex $V$ there exists a unique vertex $v_0\in V$ with only incoming edges i.e. $\exists v_0\in V$ such that $\forall w\in V\setminus\{v_0\},\ \exists k\in S$ such that $v_0=wk$. Since $S\cap S^{-1} = \emptyset$, the action of $G$ on $\Gamma(G,S)$ preserves the orientation of the edges and is free on the vertices. So $g\cdot x = x$ implies $g\cdot v_0 = v_0$ hence $g = e$.
\end{proof}

 \section{Systolic Cayley Complexes}
Consider a group $G$ with a finite triangular presentation $G = \langle S\mid R\rangle$. Then we know that $Flag(G,S)$ is a simply connected simplicial complex. We now want to know when $Flag(G,S)$ is systolic. We already know that $Flag(G,S)$ is simply connected. As it is flag, all the links of vertices are flag. So we need to check whether Lk$(v, Flag(G,S))$ is 6-large for all vertices of $Flag(G,S)$. As the action of $G$ on $Flag(G,S)$ is transitive and by isometries on the vertices, we only need to check if Lk$(e, Flag(G,S))$ is 6-large. 
For simpler notation we set $L = $ Lk$(e, Flag(G,S))$. Then the vertices of $L$ are $V(L) =  \{v[s]\mid s\in S\cup S^{-1}\}$. As we differentiate between elements in $S$ and elements in $S^{-1}$, we will call vertices $v[s]$ with $s\in S$ \textit{positive} and vertices $v[s]$ with $s\in S^{-1}$ \textit{negative}. Edges between these vertices are labeled by elements in $S$ and $E(L) = \{e[g,a]\mid g\in S\cup S^{-1} \text{ and } a\in S\setminus\{g^{-1}\} \text{ and } ga\in S\cup S^{-1}\}$, where $e[g,a]$ is the edge going from $v[g]$ to $v[ga]$. We write $v\sim w$ for two adjacent vertices $v$ and $w$ and $e = (v,w)$ for the edge $e$ between $v$ and $w$. If there is an edge from $v$ to $w$ we might also use the notation $v\rightarrow w$. To simplify the notation, we will also denote the vertex $v[s]$ with $s$ and say the vertex $s\in S$ is positive, $s\in S^{-1}$ is negative. 

We put some additional conditions on $S$ and $R$:
\begin{enumerate}
 \item If for some $a, b, c\in S$, $abc\in S$ then $ab \in S$ and $bc\in S$. 
 \item $R = \{a\cdot b\cdot c^{-1}\mid abc^{-1} =e \text{ in } G\}$, so we do not have relations of the form $a\cdot b\cdot c$ in $R$.
\end{enumerate}
We call a triangular presentation satisfying these additional conditions a \textit{restricted triangular presentation}. 
These conditions are mostly technical. They limit the possible diagonal free cycles in $L$ and by Proposition \ref{free} ensure that the action of $G$ on $Flag(G,S)$ is free. We don't know how to decide whether $L$ is 6-large without them. So what can we say about diagonal free cycles in $L$ of length 4 or 5 under these conditions?



\begin{Bemerkung}\label{exceptions}
The additional condition on $R$ implies that there are no edges from vertices in $S$ to vertices in $S^{-1}$, so from positive to negative vertices. The additional conditions on $S$ imply that every cycle $\gamma$ in $L$ which contains one of the following configurations of adjacent vertices has a diagonal. 
So if $a, b, c\in S$ with
\begin{enumerate}[label = \alph*)]
 \item $a, b, c\in V(\gamma)$ and $a\rightarrow b \rightarrow c$ then $a\rightarrow c$.
 \item $a^{-1}, b^{-1}, c^{-1}\in V(\gamma)$ and $a^{-1}\rightarrow b^{-1}\rightarrow c^{-1}$ then $a^{-1}\rightarrow c^{-1}$.
 \item $a^{-1}, b, c\in V(\gamma)$ and $a^{-1}\rightarrow b\leftarrow c$ then $a^{-1}\rightarrow c$.
 \item $a^{-1}, b^{-1},c \in V(\gamma)$ and $a^{-1}\leftarrow b^{-1}\rightarrow c$ then $a^{-1}\rightarrow c$.
\end{enumerate}
So all of these configurations of vertices cannot occur in diagonal free cycles. This leads us to the following statement about potential cycles of length 4 or 5.
\end{Bemerkung}

\begin{Lemma}\label{5-cycle}
 Every cycle of length 5 in $L$ contains a diagonal.
\end{Lemma}
\begin{proof}
Let $\gamma$ be a cycle of length 5 in $L$.

 As 5 is odd, if $\gamma$ has 5 positive or 5 negative vertices, it has a diagonal (situation a) or b) always occurs).
 
 Assume $\gamma$ contains one negative and four positive vertices. As there are only edges from negative to positive vertices, the direction of two of the edges in $\gamma$ is already determined. Every possible direction of the three other edges leads to one of the situations above (avoiding situation c) necessarily leads to situation a)). The same argument holds if $\gamma$ has one positive and four negative vertices.
 
 Assume $\gamma$ has two negative and 3 positive vertices. As each negative vertex is adjacent to at least one positive vertex, the direction of at least two edges of $\gamma$ is already determined. Every possible direction of the three other edges leads to one of the situations above (if the two negative vertices are adjacent we have situation d) otherwise situation c)). The same argument holds if $\gamma$ has two positive and three negative vertices.
\end{proof}

\begin{Lemma}\label{4-cycle}
 Let $u, v, w, x, a, b, c, d\in S$. The only cycles of length 4 in $L$ that do not contain one of the situations mentioned in Remark \ref{exceptions} are: 
 
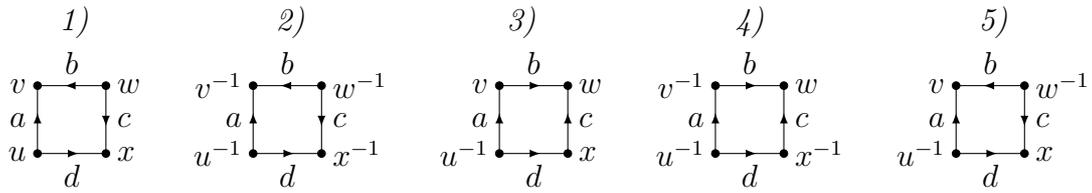
\begin{center}
\begin{tabular}{c c c c c}
1) & 2) & 3)& 4) & 5) \\
\begin{tikzpicture}

 \draw[fill] (0,0) circle [radius=0.05];
 \draw[fill] (0.9,0) circle [radius=0.05];
 \draw[fill] (0,0.9) circle [radius=0.05];
 \draw[fill] (0.9,0.9) circle [radius=0.05];
 \draw[->-, >=latex](0,0)node[left]{$u$}--node[midway, below]{$d$}(0.9,0)node[right]{$x$};
 \draw[->-, >=latex](0.9,0.9)node[right]{$w$}--node[midway, right]{$c$}(0.9,0);
 \draw[->-, >=latex](0,0)--node[midway, left]{$a$}(0,0.9)node[left]{$v$};
 \draw[->-, >=latex] (0.9,0.9)--node[midway, above]{$b$}(0,0.9);
 \end{tikzpicture}
 &
 \begin{tikzpicture}
 \draw[fill] (2.4,0) circle [radius=0.05];
 \draw[fill] (3.3,0) circle [radius=0.05];
 \draw[fill] (2.4,0.9) circle [radius=0.05];
 \draw[fill] (3.3,0.9) circle [radius=0.05];
 \draw[->-, >=latex](2.4,0)node[left]{$u^{-1}$}--node[midway, below]{$d$}(3.3,0)node[right]{$x^{-1}$};
 \draw[->-, >=latex](3.3,0.9)node[right]{$w^{-1}$}--node[midway, right]{$c$}(3.3,0);
 \draw[->-, >=latex](2.4,0)--node[midway, left]{$a$}(2.4,0.9)node[left]{$v^{-1}$};
 \draw[->-, >=latex] (3.3,0.9)--node[midway, above]{$b$}(2.4,0.9);
 \end{tikzpicture}
 &
 \begin{tikzpicture}
 \draw[fill] (5.3,0) circle [radius=0.05];
 \draw[fill] (6.2,0) circle [radius=0.05];
 \draw[fill] (5.3,0.9) circle [radius=0.05];
 \draw[fill] (6.2,0.9) circle [radius=0.05];
 \draw[->-, >=latex](5.3,0)node[left]{$u^{-1}$}--node[midway, below]{$d$}(6.2,0)node[right]{$x$};
 \draw[->-, >=latex](6.2,0)--node[midway, right]{$c$}(6.2,0.9)node[right]{$w$};
 \draw[->-, >=latex](5.3,0)--node[midway, left]{$a$}(5.3,0.9)node[left]{$v$};
 \draw[->-, >=latex] (5.3,0.9)--node[midway, above]{$b$}(6.2,0.9);
 \end{tikzpicture}
 &
 \begin{tikzpicture}
 \draw[fill] (7.9,0) circle [radius=0.05];
 \draw[fill] (8.8,0) circle [radius=0.05];
 \draw[fill] (7.9,0.9) circle [radius=0.05];
 \draw[fill] (8.8,0.9) circle [radius=0.05];
 \draw[->-, >=latex](7.9,0)node[left]{$u^{-1}$}--node[midway, below]{$d$}(8.8,0)node[right]{$x^{-1}$};
 \draw[->-, >=latex](8.8,0)--node[midway, right]{$c$}(8.8,0.9)node[right]{$w$};
 \draw[->-, >=latex](7.9,0)--node[midway, left]{$a$}(7.9,0.9)node[left]{$v^{-1}$};
 \draw[->-, >=latex] (7.9,0.9)--node[midway, above]{$b$}(8.8,0.9);
 \end{tikzpicture}
 &
 \begin{tikzpicture}
 \draw[fill] (10.9,0) circle [radius=0.05];
 \draw[fill] (11.8,0) circle [radius=0.05];
 \draw[fill] (10.9,0.9) circle [radius=0.05];
 \draw[fill] (11.8,0.9) circle [radius=0.05];
 \draw[->-, >=latex](10.9,0)node[left]{$u^{-1}$}--node[midway, below]{$d$}(11.8,0)node[right]{$x$};
 \draw[->-, >=latex](11.8,0.9)node[right]{$w^{-1}$}--node[midway, right]{$c$}(11.8,0);
 \draw[->-, >=latex](10.9,0)--node[midway, left]{$a$}(10.9,0.9)node[left]{$v$};
 \draw[->-, >=latex] (11.8,0.9)--node[midway, above]{$b$}(10.9,0.9);
 \end{tikzpicture}
 \\
 \end{tabular}

\end{center}

\end{Lemma}
\begin{proof}
 Let $\gamma$ be a cycle of length 4.
 
 Assume $\gamma$ has 4 positive vertices. In order not to be in situation Remark \ref{exceptions} a) each vertex has either two incoming or two outgoing vertices. This corresponds to the first cycle. The same argument holds if $\gamma$ has 4 negative vertices. Then we have the second cycle.
 
 Assume $\gamma$ has 1 negative and 3 positive vertices. Then the negative vertex is adjacent to two positive vertices and hence the direction of two edges is already determined. As we do not allow the configuration Remark \ref{exceptions} c) the only possible cycle is the third cycle. The same argument holds for 1 positive and 3 negative vertices, which gives the fourth cycle.
 
 Assume $\gamma$ has 2 positive and 2 negative vertices. Then if the two negative vertices are adjacent we are necessarily in situation Remark \ref{exceptions} d). So the negative vertices are not adjacent. Hence they are both adjacent to the two positive vertices and the direction of these edges is determined. This gives the fifth cycle.
\end{proof}


So to see if $L$ is 6-large we need to concentrate on the cycles of length 4 presented in Lemma \ref{4-cycle}. When do they exist? Under which conditions do they have a diagonal? The next lemma aims to answer those questions.





\begin{Lemma}\label{diagonals} The link $L$ is 6-large if and only if the following additional conditions on $S$ are satisfied:
 \begin{enumerate}[label=\arabic*)]
  \item\label{first condition} If $\exists u, w, a, b, c, d\in S$, $u\neq w$, $a\neq d$, with $ua=wb\in S$ and $ud=wc\in S$ then $\exists k\in S$ such that $w=uk$ or $ua = udk$.
  \item\label{second condition} If $\exists v,x, a, b, c, d\in S$, $v\neq x$, $a\neq b$, with $bv = cx\in S$ and $av=dx\in S$ then $\exists k\in S$ such that $v= kx$ or $ av = kbv$.
  \item\label{third condition} If $\exists u, v, x, b, c\in S$, $v\neq x$,  with $ux\in S$, $uv \in S$ and $vb= xc\in S$ then $\exists k\in S$ such that $k = uvb$ or $v = xk$.
  \item\label{fourth condition} If $\exists v, w, x, a, d\in S$, $v\neq x$, with $vw\in S$, $xw\in S$ and $dx= av\in S$ then $\exists k\in S$ such that $k = avw$ or $x=kv$.
  \item\label{fifth condition} If $\exists u, v, w, x\in S$, $v\neq x$, $u\neq w$, with $wv\in S$, $wx\in S$, $uv\in S$ and $ux\in S$ then $\exists k\in S$ such that $w=ku$ or $x = vk$.
 \end{enumerate}

\end{Lemma}

Note that $ u, v, w, x\in S$ correspond to vertices in 4-cycles in $L$ and $a, b, c, d\in S$ to edges. Also note that these conditions are all necessary as one can see in Lemma \ref{free counterexample}.
\begin{proof} We know by Lemma \ref{5-cycle} that there are no diagonal free cycles of length 5 in $L$. By Lemma \ref{4-cycle}, we know that there are only five problematic cycles of length 4. We show here under which conditions on $S$ such cycles exists and which conditions are necessary for the existence of a diagonal. If those five 4-cycles have a diagonal, all cycles of length 4 have a diagonal. So all cycles of length $<6$ have a diagonal, so $L$ is 6-large. The existence of the 4-cycle relies on two elements: we need 4 distinct vertices $u, v, w, x$ and we need the appropriate edges $a, b, c, d$ between these vertices.

\begin{enumerate}[label=\roman*)]
\item In the first cycle of length 4: the existence of the cycle is equivalent to the following statement about elements of $S$: \begin{multline*}\exists\ a, b, c, d, u, v, w, x\in S, u,v, w, x \text{ pairwise distinct}\\\text{ such that } v = ua = wb \text{ and } x = ud=wc\end{multline*}
where $u, v, w, x\in S$ are the labels on the vertices and $a, b, c, d$ are the labels on the edges. There is a diagonal if $v\sim x$ or $u\sim w$ which is equivalent to $$\exists\ k\in S: w = uk, u = wk, v = xk \text{ or } x = vk$$
where $k\in S$ is the label on the diagonal. Using the symmetries of the 4-cycle, this corresponds to condition \ref{first condition}. 

\item In the second cycle of length 4: the existence of the cycle is equivalent to the following statement about elements of $S$: \begin{multline*}\exists\ a, b, c, d, u, v, w, x\in S,u^{-1},v^{-1}, w^{-1}, x^{-1} \text{ pairwise distinct}\\\text{ such that } v^{-1}=u^{-1}a=w^{-1}b\text{ and }x^{-1} = w^{-1}c = u^{-1}d\end{multline*}where $a, b, c, d\in S$ are the labels on the edges. There is a diagonal if \linebreak $v^{-1}\sim x^{-1}$ or $u^{-1}\sim w^{-1}$ which is equivalent to $$\exists\ k\in S: w^{-1} = u^{-1}k, u^{-1} = w^{-1}k, v^{-1} = x^{-1}k \text{ or } x^{-1} = v^{-1}k$$
where $k\in S$ is the label on the diagonal. Using the symmetries of the 4-cycle this corresponds to  condition \ref{second condition}.

\item In the third cycle of length 4: the existence of the cycle is equivalent to the following statement about elements of $S$: \begin{multline*}\exists\ a, b, c, d, u, v, w, x\in S , u^{-1}, v, w, x \text{ pairwise distinct}\\\text{ such that }v = u^{-1}a, x = u^{-1}d, w = vb \text{ and }w =xc\end{multline*}where $a, b, c, d\in S$ are the labels on the edges. There is a diagonal if $v\sim x$ or $u^{-1}\sim w$ which is equivalent to $$\exists k \in S : uw = k, v = xk \text{ or } x = vk$$where $k\in S$ is the label on the diagonal. As $u^{-1}$ is a negative vertex and $w$ is a positive one, there is only one possible direction for the diagonal from $u^{-1}$ to $w$. Using the symmetry of the 4-cycle 
 this corresponds to condition \ref{third condition}.

\item In the fourth cycle of length 4: the existence of the cycle is equivalent to the following statement about elements of $S$: \begin{multline*}\exists\ a, b, c, d, u, v, w, x\in S,u^{-1},v^{-1},w, x^{-1} \text{ pairwise distinct}\\\text{ such that } v^{-1} = u^{-1}a, w = v^{-1}b= x^{-1}c \text{ and } x^{-1} = u^{-1}d\end{multline*}where $a, b, c, d\in S$ are the labels on the edges. There is a diagonal if \linebreak $v^{-1}\sim x^{-1}$ or $u^{-1}\sim w$ which is equivalent to $$\exists k \in S : uw = k, v^{-1} = x^{-1}k \text{ or } x^{-1} = v^{-1}k$$where $k\in S$ is the label on the diagonal. As $u^{-1}$ is a negative vertex and $w$ is a positive one, there is only one possible direction for the diagonal from $u^{-1}$ to $w$. Using the symmetry of the 4-cycle 
 this corresponds to condition \ref{fourth condition}.

\item In the fifth cycle of length 4: the existence of the cycle is equivalent to the following statement about elements of $S$: \begin{multline*}\exists\ a, b, c, d, u, v, w, x\in S, u^{-1},v, w^{-1}, x\text{ pairwise distinct}\\ \text{such that } v=u^{-1}a=w^{-1}b \text{ and } x = u^{-1}d = w^{-1}c\end{multline*} where $a, b, c, d\in S$ are the labels on the edges. There is a diagonal if $v\sim x$ or $u^{-1}\sim w^{-1}$ which is equivalent to $$\exists\ k\in S: w^{-1} = u^{-1}k, u^{-1} = w^{-1}k, v = xk \text{ or } x = vk$$
where $k\in S$ is the label on the diagonal. Using the symmetries of the 4-cycle 
this corresponds to condition \ref{fifth condition}.
\end{enumerate}
\end{proof}


We can now get back to the original question: when is $Flag(G,S)$ systolic?

\begin{Theorem}\label{systolic}
 Consider a group $G$ with generating set $S$, where $G$ has a finite restricted triangular presentation with respect to $S$. Then the complex $Flag(G,S)$ is a simply connected simplicial complex. It is systolic if and only if the generating set $S$ satisfies the following conditions: 
 \begin{enumerate}[label=\arabic*)]
  \item If $\exists u, w, a, b, c, d\in S$, $u\neq w$, $a\neq d$, with $ua=wb\in S$ and $ud=wc\in S$ then $\exists k\in S$ such that $w=uk$ or $ua = udk$.
  \item If $\exists v,x, a, b, c, d\in S$, $v\neq x$, $a\neq b$, with $bv = cx\in S$ and $av=dx\in S$ then $\exists k\in S$ such that $v= kx$ or $ av = kbv$.
  \item If $\exists u, v, x, b, c\in S$, $v\neq x$,  with $ux\in S$, $uv \in S$ and $vb= xc\in S$ then $\exists k\in S$ such that $k = uvb$ or $v = xk$.
  \item If $\exists v, w, x, a, d\in S$, $v\neq x$, with $vw\in S$, $xw\in S$ and $dx= av\in S$ then $\exists k\in S$ such that $k = avw$ or $x=kv$.
  \item If $\exists u, v, w, x\in S$, $v\neq x$, $u\neq w$, with $wv\in S$, $wx\in S$, $uv\in S$ and $ux\in S$ then $\exists k\in S$ such that $w=ku$ or $x = vk$.
 \end{enumerate}
 Moreover, this implies that $G$ is a systolic group.
\end{Theorem}

\begin{proof}
 It follows from Proposition \ref{triangular} and the definition of a restricted triangular presentation that $\Gamma(G,S)$ is a connected simplicial graph and $Flag(G,S)$ is a welldefined simply connected flag complex. The complex $Flag(G,S)$ is systolic if and only if the link of every vertex is flag and 6-large. 
 As $Flag(G,S)$ is a flag complex, the link of a vertex is flag. The action of $G$ on the vertices of $\Gamma(G,S) = Flag(G,S)^{(1)}$ is transitive and by isometries. So $Flag(G,S)$ is systolic if and only if the link $L =$Lk$(e,Flag(G,S))$ is 6-large. This is equivalent to the conditions given by Lemma \ref{diagonals}. 
 Since $G$ acts properly discontinuously and cocompactly on $Flag(G,S)$, the group $G$ is systolic if $Flag(G,S)$ is systolic. 
\end{proof}

We call a presentation satisfying the conditions of Theorem \ref{systolic}, a \textit{systolic presentation}.
We say a group is \textit{Cayley systolic} if it admits a systolic presentation. By Proposition \ref{free} a Cayley systolic group $G$ with systolic presentation $\langle S\mid R\rangle$ acts freely on $Flag(G,S)$. One can also note that free products of Cayley systolic groups are also Cayley systolic. More generally we do not know under which conditions amalgamated products of Cayley systolic groups are systolic or Cayley systolic.

\begin{Lemma}\label{free counterexample}
	Consider the set $S = \{a,b,c,d,u,v,w,x\}$ and the sets \begin{align*}
	& R_1 = \{uav^{-1}, wbv^{-1}, udx^{-1}, wcx^{-1}\}, \\
	& R_2 = \{bvw^{-1}, cxw^{-1}, avu^{-1}, dxu^{-1}\}, \\
	& R_3 = \{vbw^{-1}, xcw^{-1}, uxd^{-1}, uva^{-1}\}, \\
	& R_4 = \{dxu^{-1}, avu^{-1}, vwb^{-1}, xwc^{-1}\}, \\
	& R_5 = \{vua^{-1}, vwb^{-1}, xwc^{-1}, xud^{-1}\}.
	\end{align*}
	For all $i\in \{1,2,3,4,5\}$, the presentation $\langle S\mid R_i\rangle$ is a restricted triangular presentation. Additionally the presentation $\langle S\mid R_i \rangle$ satisfies all the conditions of Theorem \ref{systolic} except condition i.
\end{Lemma}

\begin{proof}
	We show this for $\langle S \mid R_3 \rangle$. The other cases can be checked in the same way. First note that $a = uv$, $b = v^{-1}w$, $c = x^{-1}w$ and $d = ux$. So the group $C_3 = \langle S\mid R_3\rangle$ is in fact $F(u,v,w,x)$ the free group on the generators $u,v,w$ and $x$. So the word problem in $C_3$ is solvable. We do the following calculations using SageMath. We check for all triples $(\alpha, \beta, \gamma)\in S^3$ that $\alpha\beta\gamma \neq 1$, $\alpha\beta\gamma\notin S$ and $\alpha\beta\gamma^{-1} = e \Leftrightarrow \alpha\cdot\beta\cdot\gamma^{-1}\in R_3$. Note that $\alpha\beta\gamma\notin S$ implies $S\cap S^{-1} = \emptyset$. So the presentation $\langle S\mid R_3\rangle$ is a restricted triangular presentation. Now we need to check the different conditions of Theorem \ref{systolic}. Using SageMath, we see that there are no elements $s_u, s_v, s_w, s_x, s_a, s_b, s_c, s_d \in S$ satisfying the hypothesis of one of the conditions \ref{first condition}, \ref{second condition}, \ref{fourth condition} and \ref{fifth condition} but the tuples $(s_u,s_v,s_x,s_b,s_c) = (u,v,x,b,c)$ and $(s_u,s_v,s_x,s_b,s_c)= (u,x,v,c,b)$  satisfy $s_vs_b=s_xs_c\in S$, $s_us_x \in S$ and $ s_us_v \in S$. We check that condition \ref{third condition} fails in at least one of those cases so we check that for all $s\in S$ we have $s\neq s_us_vs_b$ and $s_v\neq s_xs$ for at least one of those tuples.

\end{proof}

\begin{Beispiel}\label{counterexample}
The conditions in Theorem \ref{systolic} are all necessary as shown in Lemma \ref{free counterexample}. One can note that the given presentations are not systolic but the underlying groups are, since free groups are known to be systolic and even Cayley systolic with respect to the standard generating system. 
 Also note that the following presentation \begin{multline*}F_2\times F_2 = \langle a, b, c, d, \Delta_1, \Delta_2, \Delta_3, \Delta_4\mid \\ \Delta_1 = ab =ba,\ \Delta_2 = bc=cb,\ \Delta_3 = cd=dc,\ \Delta_4 = da=ad\rangle\end{multline*} satisfies conditions \ref{first condition}--\ref{fourth condition} but not condition \ref{fifth condition}. Again $F_2\times F_2$ is known to be systolic, using a construction by Elsner and Przytycki \cite{ElsPrz}. We do not know whether it is Cayley systolic. 
\end{Beispiel}

\begin{Bemerkung}\label{order} We can also interpret the conditions in Lemma \ref{diagonals} as conditions on a given order on $S$. 
 We first introduce left and right orders on $S\cup{e}$ by defining for $a, b\in S\cup\{e\}$:
\begin{enumerate}[label=\alph*)]
 \item $a\leq_L b$ if $\exists c\in S\cup\{e\}$ such that $ac = b$.
 \item $a\leq_R b$ if $\exists c\in S\cup\{e\}$ such that $ca = b$.
\end{enumerate}
These are indeed orders on $S\cup\{e\}$: as $e\in S\cup\{e\}$ they are reflexive, as there are no inverses in $S$ they are antisymmetric and by the additional condition on $S$ they are transitive.
An edge $v\rightarrow w$ in $L$ is equivalent to
\begin{enumerate}
 \item $v\leq_L w$ if $v, w\in S$ 
 \item $w^{-1}\leq_R v^{-1}$ if $v, w\in S^{-1}$
 \item $v^{-1}w\in S$ if $w\in S$ and $v\in S^{-1}$.
 \end{enumerate}
 Then the conditions on the existence of diagonals in given 4-cycles could be reformulated in the following way:
 
  \begin{enumerate}[label=\arabic*)]
  \item If $\exists u, v, w, x\in S$ pairwise distinct with $u\leq_L v \text{ and } w\leq_L v\text{ and } \linebreak u\leq_L x \text{ and } w\leq_L x$ then $u\leq_L w \text{ or } 
  x\leq_L v$.
  \item If $\exists u, v, w, x\in S$ pairwise distinct with $v\leq_R u\text{ and } x\leq_R u \text{ and }\linebreak v\leq_R w \text{ and }x\leq_R w$ then $u\leq_R w \text{ or }
  v\leq_R x$.
  \item If $\exists u, v, w, x\in S$ pairwise distinct with $uv\in S \text{ and } ux\in S\text{ and } \linebreak v\leq_L w \text{ and }x\leq_L w$ then $uw\in S\text{ or } v\leq_L x$. 
  \item If $\exists u, v, w, x\in S$ pairwise distinct with $vw\in S\text{ and } xw\in S \text{ and }\linebreak x\leq_R u \text{ and }v\leq_R u$ then $uw\in S\text{ or } v\leq_R x$.
  \item If $\exists u, v, w, x\in S$ pairwise distinct with $uv\in S \text{ and }ux\in S \text{ and }\linebreak wv\in S \text{ and }wx\in S$ then $v\leq_L x \text{ or } 
  w\leq_R u$.
 \end{enumerate}
Here the conditions are only given in terms of elements of $S$ associated to the vertices of the 4-cycles.
\end{Bemerkung}

 \section{Applications}

\subsection{Garside groups}

A group $G$ is said to be a \textit{Garside group} with \textit{Garside structure} $(G,P,\Delta)$ if it admits a submonoid $P$ with $P\cap P^{-1} = \{e\}$, called the \textit{monoid of positive elements} and a special element $\Delta\in P$ called \textit{Garside element} such that the following properties are satisfied:
\begin{enumerate}
 \item The partial order $\leq_L$ defined by $a\leq_Lb \Leftrightarrow a^{-1}b\in P$ is a lattice order, i.e. for every $a, b\in G$ there exists a unique lcm $a\vee_L b$ and a unique gcd $a\wedge_L b$ with respect to $\leq_L$ i.e. $\forall a, b\in G\ \exists! (a\vee_L b)$ such that $a\leq_L (a\vee_L b)$, $b\leq_L (a\vee_L b)$ and $\forall c\in G$ $a\leq_L c$ and $b\leq_L c$ imply $(a\vee_L b)\leq_L c$. Similarly $\forall a, b\in G\ \exists! (a\wedge_L b)$ such that $(a\wedge_L b)\leq_L a$, $(a\wedge_L b)\leq_L b$ and $\forall c\in G$ $c\leq_L a$ and $c\leq_L b$ imply $ c\leq_L (a\wedge_L b)$.
 \item The set $[e,\Delta] = \{a\in G\mid e\leq_L a\leq_L\Delta\}$, called the set of \textit{simple elements}, generates $P$.
 \item Conjugation by $\Delta$ preserves $P$ i.e. $\Delta^{-1}P\Delta = P$.
 \item For all $x\in P\setminus\{e\}$ one has $$\lVert x\rVert = sup\{k\in \mathbb{N}\mid \exists a_1,\dots, a_k\in P\setminus\{e\} \text{ such that } x = a_1\cdots a_k\}<\infty.$$
\end{enumerate}

A Garside structure $(G,P,\Delta)$ is said to be of \textit{finite type} if the set of simple elements $[e,\Delta]$ is finite. A group $G$ is said to be a \textit{Garside group of finite type} if it admits a Garside structure of finite type. Elements $x\in P\setminus\{e\}$ with $\lVert x \rVert = 1$ are called \textit{atoms}. The set of atoms also generates $P$.

\begin{Bemerkung}
 The monoid $P$ also induces a partial order $\leq_R$ which is invariant under right multiplication. We define $a\leq_Rb\Leftrightarrow ba^{-1}\in P$. It follows from the properties of $G$ that $\leq_R$ is also a lattice order, that $P$ is the set of elements $a$ such that $e\leq_Ra$ and that the simple elements are the elements $a$ such that $e\leq_R a\leq_R\Delta$. We denote $a\vee_R b$ the lcm and $a\wedge_R b $ the gcd with respect to $\leq_R$. So $\{a\in G\mid e\leq_L a\leq_L\Delta\}=\{a\in G\mid e\leq_R a\leq_R\Delta\}$, we also say that $\Delta$ is \textit{balanced}.
\end{Bemerkung}

\begin{Beispiel}
	Spherical Artin groups are Garside groups, in particular braid groups are Garside. Torus knot groups $\langle x, y\mid x^p = y^q\rangle$ with $p, q>1$ are Garside groups with Garside element $\Delta = x^p = y^q$. The fundamental group of the complement of $n$ line through the origin in $\mathbb{C}^2$ $$\langle x_1, \dots, x_n\mid x_1x_2\dots x_n = x_2\dots x_nx_1 = \dots = x_nx_1\dots x_{n-1}\rangle$$ is also a Garside group with Garside element $\Delta = x_1x_2\dots x_n$.
\end{Beispiel}

\begin{Lemma}\label{Garside presentation}
 Let $G$ be a Garside group with Garside element $\Delta$ and set of simple elements $S$. Then $G = \langle S\setminus\{e\}\mid s\cdot t = st\ \forall s, t\in S\ such\ that\ st\in S\rangle$ is a restricted triangular presentation of $G$.
\end{Lemma}

\begin{proof} 
This  presentation is a direct consequence of Theorem 6.1 in \cite{DehPar}. It is a restricted triangular presentation as all relations are of the form $a\cdot b =c$ for some $a,\ b,\ c\in S$ and $S$ is the set of simple elements of a Garside group.
\end{proof}
We call this presentation a \textit{Garside presentation} of $G$ with Garside element $\Delta$. We can now state one of our main results.

\begin{Theorem}\label{Garside} Let $G$ be a Garside group of finite type with Garside element $\Delta$ and non trivial simple elements $S$. 
Then $Flag(G,S)$ is systolic if and only if for all $a, b\in S$, $a\wedge_Lb\in \{e,a,b\}$ and $a\wedge_R b\in \{e,a,b\}$. In particular if $Flag(G,S)$ is systolic then so is $G$. 

\end{Theorem}

\begin{proof}
By Lemma \ref{Garside presentation}, the Garside presentation of $G$ with generating set $S$ is a restricted triangular presentation.
By Proposition \ref{triangular}, $Flag(G,S)$ is well-defined. 
By Theorem \ref{systolic}, $Flag(G,S)$ is systolic if and only if the conditions \ref{first condition}--\ref{fifth condition} are satisfied.

Assume that for all $a, b\in S$ $a\wedge_Lb, a\wedge_R b\in\{e,a,b\}$. Then

\ref{first condition} If $\exists u, w, a, b, c, d\in S, u\neq w, a\neq b$ with $ua=wb\in S$ and $ud = wc\in S$, then $u\leq_L ua$ and $u\leq_L ud$ so $ua\wedge_L ud\neq e$, so either $ua\wedge_L ud = ua$ or $ua\wedge_L ud = ud$, say $ua\wedge_Lud=ud$ so $ua = udk$ for some $k\in S$.

\ref{second condition} If $\exists v, x, a, b, c, d\in S, v\neq x, a\neq b$ with $bv=cx\in S$ and $av=dx\in S$ then $v\leq_R bv$ and $v\leq_R av$ then $av\wedge_R bv\neq e$, so either $av\wedge_R bv = av$ or $av\wedge_R bv = bv$, say $av\wedge_R bv = bv$ so $av = kbv$ for some $k\in S$.

\ref{third condition} If $\exists u, v, x, b, c\in S, v\neq x$ with $ux\in S, uv \in S$ and $vb=xc \in S$ then $u\leq_L ux$ and $u\leq_L uv$ so $uv\wedge_L ux\neq e$, so either $uv\wedge_L ux=uv$ or $uv\wedge_L ux=ux$, say $uv\wedge_L ux= ux$ so $uv = uxk$ for some $k\in S$ and then $v = xk$.

\ref{fourth condition} If $\exists v, w, x, a, d \in S$, $v\neq x$ with $vw\in S, xw\in S$ and $dx=av\in S$ then $w\leq_R vw$ and $w\leq_R xw$ so $vw\wedge_R xw\neq e$, so either $vw\wedge_R xw = vw$ or $vw\wedge_R xw = xw$, say $xw\wedge_R vw = vw$ so $xw= kvw$ for some $k\in S$ and then $x = kv$.

\ref{fifth condition} If $\exists u, v, w, x\in S, v\neq x, u\neq w$ with $wv\in S, wx\in S, uv\in S$ and $ux\in S$ then $u\leq_L uv$ and $u\leq_L ux$ so $uv\wedge_L ux\neq e$, so either $uv\wedge_L ux = uv$ or $uv\wedge_L ux = ux$,  say $uv\wedge_L ux = uv$ so $uv= uxk$ for some $k\in S$ and then $v = xk$.

So if for all $a, b\in S$, $a\wedge_Rb, a\wedge_Lb\in\{e,a,b\}$, the conditions \ref{first condition}--\ref{fifth condition} of Theorem \ref{systolic} are satisfied and so $Flag(G,S)$ is systolic, which directly implies that $G$ is sytolic.

We now show the other implication. First assume $\exists a, b\in S$, $a\neq b$, with $a\wedge_Lb = c$ for some $c\in S\setminus\{a, b\}$, i.e $c\neq e, a ,b$. Then $\exists k_a, k_b, r_a, r_b\in S$ with $a = ck_a$, $b=ck_b$ and $\Delta = k_ar_a=k_br_b$. Then $c, k_a, k_b, r_a, r_b\in S$ and $k_a\neq k_b$, $ck_a =a\in S$, $ck_b = b\in S$ and $\Delta = k_ar_a = k_br_b\in S$. But $ck_ar_a = c\Delta \notin S$ and $\nexists k\in S$ with $k_a = k_bk$ since $c = a\wedge_L b \neq b$, similarly $\nexists k\in S$ with $k_b = k_ak$ since $c = a\wedge_L b \neq a$. So condition \ref{third condition} of Theorem \ref{systolic} fails. So $Flag(G,S)$ is not systolic.
Finally assume $\exists a, b\in S$, $a\neq b$, with $a\wedge_R b=c$ for some $c\in S\setminus\{a, b\}$, i.e. $c\neq e, a, b$. Then $\exists k_a, k_b, r_a, r_b\in S$ with $a = k_ac$, $b=k_bc$ and $\Delta = r_ak_a=r_bk_b$. Then $c, k_a, k_b, r_a, r_b\in S$ and $k_a\neq k_b$, $k_ac =a\in S$, $k_bc = b\in S$ and $\Delta = r_ak_a = r_bk_b\in S$. But $r_ak_ac = \Delta c\notin S$ and $\nexists k\in S$ with $k_a = kk_b$ or $k_b = kk_a$ since $c =a\wedge_R b \neq a, b$. So condition \ref{fourth condition} of Theorem \ref{systolic} fails. So $Flag(G,S)$ is not systolic.
So if there exist $a, b\in S$ with $a\wedge_Lb \notin\{e,a,b\}$ or $a\wedge_Rb\notin\{e,a,b\}$, the complex $Flag(G,S)$ is not systolic.
\end{proof}

\begin{Beispiel}\label{general case}
	Let $x_1,\dots, x_n$ be $n$ letters and let $m$ be a positive integer. We define $$\prodG(x_1,\dots, x_p;m) = \underbrace{x_1x_2\dots x_px_1x_2\dots}_{m}.$$ and $\prodG(x_1,\dots,x_p;0) = e$. Consider the group \begin{equation*}\begin{split}G_{n,m} = \langle x_1,\dots, x_n\mid &\prodG(x_1,\dots,x_n;m) = \prodG(x_2,\dots,x_{n}, x_1;m)=\dots\\ &=\prodG(x_n, x_1,\dots,x_{n-1};m)\rangle.\end{split}\end{equation*} By Proposition 5.2 in \cite{DehPar}, this is a Garside group with Garside element \linebreak $\Delta_{n,m} = \prodG(x_1,\dots,x_n;m)$. When considering all indices modulo $n$, we can write the set of simple elements as $$ S = \{\prodG(x_i, \dots, x_{i+n};k) \mid 0\leq k \leq m \text{ and } 1\leq i \leq n\}.$$
	
	In particular, for $n = 1$, we have $G_{1,m} = \langle x_1 \rangle$ with Garside element $x_1^m$ and the simple elements are $S = \{x_1^i \mid 0\leq i \leq m\}$. Note that if $m=n =1$ we have $x_1 = \Delta_{1,1}$ and $S = \{1, x_1\}$. Also if $n>1$ we can assume $m>1$.
	
	More generally for some positive intergers $p$, $n_1, \dots, n_p$, $m_1,\dots, m_p$, the product $$ G = (\ast_{i=1}^p G_{n_i,m_i})/(\Delta_{n_i,m_i}=\Delta_{n_j,m_j} \forall i, j)$$ 
	is a Garside group with Garside element $\Delta = \Delta_{n_1,m_1} = \dots = \Delta_{n_p,m_p}$. We would like to remark that if $p >1$ and $n_i = m_i = 1$ for some $i$, say $i = p$, we have $(\ast_{i=1}^p G_{n_i,m_i})/(\Delta_{n_i,m_i}=\Delta_{n_j,m_j} \forall i, j) \cong (\ast_{i=1}^{p-1} G_{n_i,m_i})/(\Delta_{n_i,m_i}=\Delta_{n_j,m_j} \forall i, j)$. 
	So we can assume that if $p>1$, we have $m_i \geq 2$ for all $i\in\{1,\dots,p\}$. The next theorem shows that these are the only Garside groups with systolic Garside presentation. 
\end{Beispiel}

 \begin{Theorem}\label{systolic Garside} Let $G$ be a Garside group of finite type.  Then $G$ has a systolic Garside presentation if and only if $G = (\ast_{i=1}^p G_{n_i,m_i})/(\Delta_{n_i,m_i}=\Delta_{n_j,m_j} \forall i, j)$ for some positive integers $p, n_1,\dots, n_p$ and $m_1,\dots, m_p$.
  
 \end{Theorem}

\begin{proof} 
 We start by showing that if $G = (\ast_{i=1}^p G_{n_i,m_i})/(\Delta_{n_i,m_i}=\Delta_{n_j,m_j} \forall i, j)$ for some positive integers $p, n_1,\dots, n_p, m_1,\dots, m_p$ the group $G$ has a systolic Garside presentation. 
 We start with the case $p = 1$. If $n = m =1$, we have $S = \{x_1\}$ so the Garside presentation is systolic. Otherwise $G=G_{n,m}$ for some positive integers $n$ and $m$, $2\leq m$ and the Garside element is $\Delta= \Delta_{n,m} = \prodG(x_1,\dots,x_n;m)$. For simpler notation we always consider the index $i$ modulo $n$. The set of simple elements is $$ S = \{\prodG(x_i, \dots, x_{i+n};k) \mid 1\leq i\leq n \text{ and } 0\leq k \leq m\}.$$Then for $0\leq k \leq l \leq m$, we have \begin{multline*}\prodG(x_i, \dots, x_{i+n};k)\wedge_L\prodG(x_j, \dots, x_{j+n};l) \\ =  \begin{cases} e &\mbox{if } i\neq j \text{ and } l<m \\ \prodG(x_i, \dots, x_{i+n};k) &\mbox{if } i=j \text{ or } l=m.\end{cases}\end{multline*} Similarly for $0\leq k \leq l\leq m$, we have \begin{multline*}\prodG(x_i, \dots, x_{i+n};k)\wedge_R\prodG(x_j, \dots, x_{j+n};l) \\ = \begin{cases} e &\mbox{if } i+k\not\equiv j+l \text{ and } l<m\\ \prodG(x_i, \dots, x_{i+n};k)&\mbox{if } i+k\equiv j+l \text{ or } l=m.\end{cases}\end{multline*} So by Theorem \ref{Garside} the Garside presentation of $G_{n,m}$ is systolic.
 
 Now consider the case $p> 1$. Then the element $\Delta = \Delta_{n_1,m_1}=\dots=\Delta_{n_p,m_p}$ is the Garside element of $G$. Let $S$ be the set of simple elements of $G$ and $S_i$ the set of simple elements of $G_{n_i,m_i}$ for $i=1,\dots,p$. Then $S=\sqcup_{i=1}^p (S_i\setminus\{\Delta_{n_i,m_i}, e\})\sqcup\{\Delta,e\}$ is a partition of the set of simple elements. For every $i, j\in\{1,\dots,p\}$, $S_i$ satisfies $s\wedge_Lt, s\wedge_R t\in \{e,s,t\}$ for all $s,t\in S_i$ and we have $s\wedge_Rt=s\wedge_Lt=e$ if $s\in S_i\setminus\{\Delta\}, t\in S_j\setminus\{\Delta\}$, $i\neq j$. So for all $s,t\in S$ we have $s\wedge_Rt, s\wedge_Lt\in\{s,t,e\}$. Hence by Theorem \ref{Garside}, $G$ has a systolic Garside presentation.

Now assume $G$ is a Garside group with systolic Garside presentation. Let $\Delta$ be the Garside element, $P$ the monoid of positive words, $S$ the set of simple elements and $A$ the set of atoms of $P$. So by Theorem \ref{Garside}, we have for all $s,t\in S$, $s\wedge_Lt,\ s\wedge_Rt\in\{e,s,t\}$. First note that if $\Delta \in A$ we have $S = \{e, \Delta\}$, $A = \{\Delta\}$, $G = \mathbb{Z}$ and $P = \mathbb{N}$. Hence we can write $G$ as $(\ast_{i=1}^p G_{n_i,m_i})/(\Delta_{n_i,m_i}=\Delta_{n_j,m_j} \forall i, j)$ with $p=1$ and $n_1 =m_1 =1$. So we can now assume that $\Delta \notin A$.

We start with showing that for all $a\in A$ there exists a unique $\xi(a)\in A$ such that $a\xi(a)\in S$. As $\Delta\notin A$, such a $\xi(a)\in A$ exists. Assume it is not unique, so let $a_1, a_2\in A$ with $aa_1, aa_2 \in S$. Since $aa_1\wedge_L aa_2 \in\{e, aa_1, aa_2\}$ and $a, a_1, a_2$ are atoms, $aa_1=aa_2$ and hence $a_1=a_2$. Similarly for all $a\in A$ there exists a unique $\rho(a) \in A$ such that $\rho(a)a\in S$. In particular for all $a\in A$, $\rho(\xi(a)) = a =\xi(\rho(a))$. So the map $\xi : A\rightarrow A$ is bijective with inverse map $\rho:A\rightarrow A$. As $A$ is finite, the map $\xi$ is a permutation of $A$. So the orbits of $\xi$ form a partition of $A$ and $\xi$ can be written as a product of cycles of disjoint support, $$\xi = (a_{1,1},a_{1,2},\dots, a_{1,n_1})(a_{2,1},\dots,a_{2,n_2})\cdots (a_{p,1},\dots,a_{p,n_p}).$$ Then for any $1\leq i \leq p$ there exists $m_i \geq 2$ such that $\Delta = \prodG(a_{i,1},\dots,a_{i,n_i};m_i)$. 
Since $\Delta$ is balanced we have \begin{multline*}\prodG(a_{i,1},\dots,a_{i,n_i};m_i)=\prodG(a_{i,2},a_{i,3},\dots,a_{i,n_i},a_{i,1};m_i)\\=\dots=\prodG(a_{i,{n_i}},a_{i,1}, \dots,a_{i,n_i-1};m_i).\end{multline*} The set $A_i =\{ a_{i,1},\dots,a_{i,n_i}\}$ corresponds to the atoms of $G_{n_i,m_i}$. Hence $$G = (\ast_{i=1}^p G_{n_i,m_i})/(\Delta_{n_i,m_i}=\Delta_{n_j,m_j} \forall i, j).$$
\end{proof}



\begin{Korollar}
	\begin{enumerate}
		\item The group $$G_{n,n} = \langle x_1,\dots,x_n\mid x_1x_2\dots x_n = x_2x_3\dots x_nx_1=\dots=x_nx_1\dots x_{n-1}\rangle,$$ which is the fundamental group of the complement of $n$ lines through the origin in $\mathbb{C}^2$, has a systolic Garside presentation with respect to the Garside element $\Delta = x_1x_2\dots x_n$.
 \item Consider $n$ positive integers $p_1,\dots,p_n$, $p_i\geq 2$. The group $$G = \langle x_1,\dots, x_n\mid x_1^{p_1}=x_2^{p_2}=\dots=x_n^{p_n}\rangle$$ has a systolic Garside presentation with respect to the Garside element \linebreak $\Delta = x_1^{p_1}$. 
 In particular torus knot groups $\langle x, y\mid x^p=y^q\rangle$ have a systolic Garside presentation with respect to the Garside element $\Delta = x^p$.
\end{enumerate}
\end{Korollar}

\begin{Bemerkung}
 In \cite{DehPar}, Example 5 mentions a generalization of the groups $G_{n,m}$. Let $p, n, m \in \mathbb{N}$, $2\leq m$, $2\leq p \leq n$. It claims that \begin{equation*}\begin{split}K_{n,p,m} = \langle x_1 ,x_2, \dots, x_n\mid \prodG(x_1,\dots,x_p; m) &= \prodG(x_2,\dots,x_{p+1}; m)=\dots \\ &=\prodG(x_{n-p+1},\dots,x_{n}; m)\\ &=\prodG(x_{n-p+2},\dots,x_{n},x_1; m)=\dots \\ &=\prodG(x_{n},x_1,\dots,x_{p-1}; m)\rangle \end{split} \end{equation*} is a Garside group by \cite{DehPar} Proposition 5.2. But this is a wrong application of \cite{DehPar} Proposition 5.2, as one can see by considering for example the case $n=5$, $p=3$ and $m=4$ or more generally $m = p+1$. So the question of whether $K_{n,p,m}$ is a Garside group when $p\neq n$ remains open.
\end{Bemerkung}

Question: For $k\geq 2$, the group $G_k = \langle a, b\mid b^k = aba\rangle$ is Garside with Garside element $\Delta = b^{k+1} = abab= baba$. The elements $b^k$ and $bab$ are both simple elements, but $b^k\wedge_L bab = b\notin\{e, b^k, bab\}$. So it does not satisfy the conditions of Theorem \ref{Garside}. 
 Is it systolic? Is there another Garside structure on this group for which it has a systolic Garside presentation?

\begin{Bemerkung}[Restrictions on systolicity in Garside groups] Consider a Garside group $G$ with Garside element $\Delta$. Then $\Delta^k$ is in the center of $G$ for some positive integer $k$. Let $S$ be the set of simple elements. Suppose there is some balanced element $\delta\in S\setminus\{\Delta, e\}$. Let $T=\{a\in G\mid 1\leq_L a\leq_L \delta\}=\{a\in G\mid 1\leq_R a\leq_R \delta\}$. Let $a\in T$ be an atom and suppose $\delta\notin\langle a\rangle$. Then $\delta^l$ is in the center of the subgroup of $G$ generated by $T$ for some positive integer $l$. If $\langle T \rangle \neq G$, we have $\langle a, \delta^l, \Delta^k\rangle \cong \mathbb{Z}^3$. This implies in particular that $G$ is not systolic.
 
\end{Bemerkung}

 
 
 
 
 
 

\subsection{Artin groups}
 Recall the notation $\left[xyx\dots\right]_k = \underbrace{xyx\dots}_{k}\ $ and $\left[\dots xyx\right]_k = \underbrace{\dots xyx}_{k}\ $ for some $k\in\mathbb{N}$. Given a finite labeled simple graph $\Gamma$, the \textit{Artin group associated to $\Gamma$} is given by \begin{multline*}A_{\Gamma} = \langle s_v, v\in V\mid \left[s_vs_ws_v\dots\right]_{m_e} = \left[s_ws_vs_w\dots\right]_{m_e} \\ \text{ for all edges } e = (v,w) \text{ with label } m_e\rangle.\end{multline*}
For $n\in\mathbb{N}_{\geq2}$, the \textit{dihedral Artin group} $DA_n$ is the Artin group defined by the graph \begin{tikzpicture}  \draw[fill] (0,0) circle [radius=0.05];
 \draw[fill] (1,0) circle [radius=0.05];
 \draw (0,0)node[left]{$a$}-- node[above]{$n$}(1,0)node[right]{$b$}; \end{tikzpicture}. So $DA_n =\langle a, b\mid \left[aba\dots\right]_n = \left[bab\dots\right]_n \rangle$. 
\newline The \textit{Artin monoid associated to $\Gamma$} is given by \begin{multline*}
A_{\Gamma}^+ = \langle s_v, v\in V\mid \left[s_vs_ws_v\dots\right]_{m_e} = \left[s_ws_vs_w\dots\right]_{m_e} \\ \text{ for all edges } e = (v,w) \text{ with label } m_e\rangle^+.\end{multline*} By \cite{Paris1} Theorem 1.1, the canonical homomorphism $\iota: A_{\Gamma}^+ \hookrightarrow A_{\Gamma}$ is an injection.

\begin{Korollar}\label{dihedral}
The dihedral Artin group $DA_n$ 
is Cayley systolic for all $n\in\mathbb{N}_{\geq2}$.
\end{Korollar}

\begin{proof}
The dihedral Artin group $DA_n$ corresponds to the Garside group $G_{2,n}$ with Garside element $\Delta = \left[aba\dots\right]_{n}$. So by Theorem \ref{systolic Garside} it is Cayley systolic.  
\end{proof}

 

Recall the following definitions: An \textit{orientation} on a simple graph $\Gamma$ is an assignement $o(e)$ for each edge $e\in E(\Gamma)$ where $o(e)$ is a set of one or two endpoints of $e$. An edge with both endpoints assigned is \textit{bioriented}. The startpoint $i(e)$ is an assignement of one or two startpoints of $e$, which is consistent with the choice of $o(e)$. If $o(e)$ consists of one point, $i(e)$ consists of one point such that $e = (i(e),o(e))$, if $o(e)$ consists of two points then so does $i(e)$.
We say that a cycle $\gamma$ is \textit{directed} if for each $v\in\gamma$ there is exactly one edge $e\in\gamma$ with $v\in o(e)$. A cycle is \textit{undirected} if it is not directed.
We say that a 4-cycle $\gamma = (a_1,a_2,a_3,a_4)$ is \textit{misdirected} if 
$a_2\in o(a_1,a_2), a_2\in o(a_2,a_3), a_4\in o(a_3,a_4)$ and $a_4\in o(a_4,a_1)$.

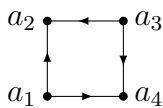
\begin{tikzpicture}

 \draw[fill] (0,0) circle [radius=0.05];
 \draw[fill] (1,0) circle [radius=0.05];
 \draw[fill] (0,1) circle [radius=0.05];
 \draw[fill] (1,1) circle [radius=0.05];
 \draw[->-, >=latex](0,0)node[left]{$a_1$}--(1,0)node[right]{$a_4$};
 \draw[->-, >=latex](1,1)node[right]{$a_3$}--(1,0);
 \draw[->-, >=latex](0,0)--(0,1)node[left]{$a_2$};
 \draw[->-, >=latex] (1,1)--(0,1);
 \end{tikzpicture}

\begin{Lemma}\label{dual presentation}
Let $\Gamma$ be a labeled simple graph with an orientation $o$ such that an edge is bioriented if and only if it has label $2$. Assume that every 3-cycle in $\Gamma$ is directed. Let $V(\Gamma) = \{v_1,\dots,v_n\}$. Consider the set \linebreak $S = \{ x_1, x_2, \dots, x_n\}\cup \{\Delta_e, t^e_1,t^e_2,\dots, t^e_{m_e-2}\mid e\in E(\Gamma)\text{ with label } m_e \}$. For each $e\in E(\Gamma)$ with label $m_e\geq 3$ and with $i(e) = v_i$ and $o(e) = v_j$, we consider the set $R_e = \{ x_ix_j\Delta_e^{-1}, x_jt^e_1\Delta_e^{-1}, t^e_1t^e_2\Delta^{-1}_e,\dots, t_{m_e-3}^et_{m_e-2}^e\Delta_e^{-1},t^e_{m_e-2}x_i\Delta_e^{-1}\}$. For each $e\in E(\Gamma)$ with label $m_e = 2$ and with $o(e) = i(e)=\{v_i,v_j\}$ we consider the set $R_e = \{x_ix_j\Delta_e^{-1}, x_jx_i\Delta_e^{-1}\}$. Let $R=\bigcup_{e\in E(\Gamma)} R_e$. Then the presentation $\langle S\mid R \rangle$ is a restricted triangular presentation of $A_{\Gamma}$. We call this the dual presentation  of $A_{\Gamma}$ with orientation $o$.
\end{Lemma}
\begin{proof}
	The standard presentation of $A_{\Gamma}$ is \begin{multline*}A_{\Gamma} = \langle x_1, \dots, x_n\mid \left[x_ix_jx_i\dots\right]_{m_e} = \left[x_jx_ix_j\dots\right]_{m_e} \\ \text{ for all edges } e = (v_i,v_j) \text{ with label } m_e\rangle. \end{multline*}
We first see that the dual presentation is indeed a presentation of $A_{\Gamma}$. Let $e\in E(\Gamma)$ with $m_e = 2$, $i(e) = o(e) = \{v_i,v_j\}$. Then the standard presentation states $x_ix_j = x_jx_i$. In the dual presentation the relations $R_e$ imply $\Delta_e = x_ix_j$ and $\Delta_e = x_jx_i$ and hence $x_ix_j=x_jx_i$.  Let $e\in E(\Gamma)$ with $m_e\geq 3$, $i(e) = v_i$ and $o(e)=v_j$. For $k \in\{1,\dots,m_e-2\}$, the relations $R_e$ imply on one hand \linebreak $t^e_k =( \left[\dots x_jx_ix_j\right]_{k})^{-1}\left[\dots x_jx_ix_j\right]_{k+1}$ and on the other hand \linebreak $t^e_k = \left[x_ix_jx_i\dots\right]_{m_e-k}(\left[x_ix_jx_i\dots\right]_{m_e-k-1})^{-1}$. This implies the relation \linebreak $\left[x_ix_jx_i\dots\right]_{m_e} = \left[x_jx_ix_j\dots\right]_{m_e}$. Conversely $\left[x_ix_jx_i\dots\right]_{m_e} = \left[x_jx_ix_j\dots\right]_{m_e}$ implies $x_ix_j = \left[x_jx_ix_j\dots\right]_{m_e}(\left[x_ix_jx_i\dots\right]_{m_e-2})^{-1} = x_jt^e_1$. For $ k\in\{1,\dots, m_e-3\}$ we have \begin{align*}
t^e_kt^e_{k+1} &= ( \left[\dots x_jx_ix_j\right]_{k})^{-1}\left[\dots x_jx_ix_j\right]_{k+1}( \left[\dots x_jx_ix_j\right]_{k+1})^{-1}\left[\dots x_jx_ix_j\right]_{k+2} \\ &= ( \left[\dots x_jx_ix_j\right]_{k})^{-1}\left[\dots x_jx_ix_j\right]_{k+2} \\ &=  x_ix_j.\end{align*} Finally when $k = m_e-2$ we have \begin{align*}t^e_{m_e-2}x_i &= (\left[\dots x_jx_ix_j\right]_{m_e-2})^{-1}\left[\dots x_jx_ix_j\right]_{m_e-1}x_i \\ &= (\left[\dots x_jx_ix_j\right]_{m_e-2})^{-1}\left[\dots x_ix_jx_i\right]_{m_e}\\ &=(\left[\dots x_jx_ix_j\right]_{m_e-2})^{-1}\left[\dots x_jx_ix_j\right]_{m_e}\\ &= x_ix_j.\end{align*} 

Let $a,b,c\in S$. To see that the dual presentation is a restricted triangular presentation we check that $abc\neq e$, $abc\notin S$ and $abc^{-1}=e \Rightarrow abc^{-1}\in R$. Note that $abc\notin S$ implies $S\cap S^{-1} = \emptyset$. We consider the following map $\xi: S \rightarrow \mathbb{Z}$ defined by $\xi(\Delta_e)=2 $ and $\xi(t^e_i) = 1$ for all $e\in E(\Gamma)$, $i\in\{1,\dots, m_e-2\}$ and $\xi(x_i) = 1$ for $i\in\{1,\dots, n\}$. As $\xi(a)+\xi(b) -\xi(c) = 0$ for all $a b c^{-1}\in R$, the map extends to a homomorphism $\xi: A_{\Gamma} \rightarrow \mathbb{Z}$. 
 For any $a,b, c\in S$, we have $\xi(a) + \xi(b) \geq 2$ and $\xi(abc) = \xi(a) + \xi(b)+\xi(c) \geq 3$ so it follows that $ab \neq e$, $abc\neq e$ and $abc\notin S$. Also for any $a, b, c\in S$ such that $abc^{-1}=e$ we have $\xi(abc^{-1}) = 0$, which implies $\xi(c) = \xi(a) + \xi(b)$ so we necessarily have $c = \Delta_e$ for some $e\in E(\Gamma)$ and $a,b \in S\setminus\{\Delta_e, e\in E(\Gamma)\}$. 
 So fix $e\in E(\Gamma)$ and $c = \Delta_e = x_ix_j$. Let $S_0 = S\setminus\{\Delta_e, e\in E(\Gamma)\}$. We need to verify for all $(a,b)\in S_0\times S_0$ that if $abc^{-1} = e$ we have $a b c^{-1}\in R$. 
 
 Our proof relies on the following property of $A_{\Gamma}^+$: Let $\alpha = x_{i_1}x_{i_2}\dots x_{i_l}$ and $\beta = x_{j_1}x_{j_2}\dots x_{j_k}$ be two words on $x_1,\dots, x_n$. They represent the same element in $A_{\Gamma}^+$ if and only if we can 
 transform $\alpha$ into $\beta$ using a finite number of transformations of the form $u\cdot \left[x_px_qx_p\dots\right]_{m_f}\cdot u' = u\cdot \left[x_qx_px_q\dots\right]_{m_f}\cdot u'$ for some $f = (v_p,v_q)\in E(\Gamma)$. If $\alpha$ does not contain a subword of this form, no transformation is possible and the expression is unique.
 
 Let $a = x_k$ and $b=x_l$ for some $k, l\in \{1,\dots,n\}$. Then $ab = c$ implies $x_kx_l = x_ix_j$ in $A_{\Gamma}$ hence the equality also holds in $A_{\Gamma}^+$. If $m_e\geq 3$ this necessarily implies $k = i$ and $l=j$. If $m_e = 2$, this implies either $k = i$ and $l=j$ or $k=j$ and $l=i$. In all of these cases the corresponding relation is in $R$.
 
 Let $a = x_k$ and $b = t^f_l$ with $k\in \{1,\dots,n\}$, $f=(v_p,v_q) \in E(\Gamma)$, \linebreak $m_f\geq 3$ and $1\leq l \leq m_f-2$. Set $l' = m_f-l-1$, so $1\leq l'\leq m_f-2$. Then $ab=c$ implies $x_kt^f_l = x_ix_j$ so $x_k\left[x_px_qx_p\dots\right]_{m_f-l}(\left[x_px_qx_p\dots\right]_{m_f-l-1})^{-1} = x_ix_j$ and hence $x_k\left[x_px_qx_p\dots\right]_{l'+1} = x_ix_j\left[x_px_qx_p\dots\right]_{l'}$ in $A_{\Gamma}^+$. As all 3-cycles are directed, $x_i, x_j, x_k$ cannot commute with both $x_p$ and $x_q$. So the last letter on the left hand side and on the right hand side is different and is either $x_p$ or $x_q$. The only way to change this last letter is to apply the relation  $\left[x_px_qx_p\dots\right]_{m_f} = \left[x_qx_px_q\dots\right]_{m_f}$. So we need $k=q$ and $l'=m_f-2$ so $l=1$. This implies the equality $x_px_q=x_ix_j$ so $p = i$ and $j=q$. Hence $f =e$, $a = x_j$, $ b = t^e_1$ and the corresponding relation is in $R$.
 
 Let $a = t^f_l$ and $b=x_k$ with $k\in \{1,\dots,n\}$, $f=(v_p,v_q) \in E(\Gamma)$, $m_f\geq 3$ and \linebreak $1\leq l \leq m_f-2$. Then $ab=c$ implies $t^f_lx_k = x_ix_j$ \linebreak so $( \left[\dots x_qx_px_q\right]_{l})^{-1}\left[\dots x_qx_px_q\right]_{l+1}x_k = x_ix_j$ so $\left[\dots x_qx_px_q\right]_{l+1}x_k = \left[\dots x_qx_px_q\right]_{l}x_ix_j$ in $A_{\Gamma}^+$. 
  As all 3-cycles are directed, $x_i, x_j, x_k$ cannot commute with both $x_p$ and $x_q$. The first letter on the left hand side and on the right hand side is different and is either $x_p$ or $x_q$. The only way to change this first letter is to use the relation  $\left[x_px_qx_p\dots\right]_{m_f} = \left[x_qx_px_q\dots\right]_{m_f}$. So we need $k=p$ and $l=m_f-2$. This implies $x_px_q=x_ix_j$ hence $ i =p$ and $j=q$. So $f = e$, $a= t^e_{m_e-2}$ and $b=x_i$ and the corresponding relation is in $R$.
 
 Finally let $a = t^f_k$ and $b = t^g_l$ with $f = (v_p, v_q), g = (v_r,v_s) \in E(\Gamma)$, $m_f, m_g\geq 3$, $1\leq k\leq m_f-2$ and $1\leq l\leq m_g-2$. Set $l' = m_g-l-1$. Then $ab=c$ implies \linebreak $t^f_kt^g_l = x_ix_j$ 
  so $\left[\dots x_qx_px_q\right]_{k+1}\left[x_rx_sx_r\dots\right]_{l'+1} =  \left[\dots x_qx_px_q\right]_{k}x_ix_j\left[x_rx_sx_r\dots\right]_{l'}$ in $A_{\Gamma}^+$. As all 3-cycles are directed, $x_r$ and $x_i$ cannot commute with both $x_p$ and $x_q$. Similarly $x_q$ and $x_j$ cannot commute with both $x_r$ and $x_s$. So the first letter on the left and on the right hand side is different and is either $x_p$ or $x_q$. Similarly the last letter on the left and on the right hand side is different and is either $x_r$ or $x_s$. So we need to apply the relations $\left[x_px_qx_p\dots\right]_{m_f} = \left[x_qx_px_q\dots\right]_{m_f}$ and $\left[x_rx_sx_r\dots\right]_{m_g} = \left[x_sx_rx_s\dots\right]_{m_g}$. This requires $r=p$ and $s=q$ and  $m_f\leq k+l'+2$. If $k+l'+2 > m_f$, we can apply the relation $\left[x_px_qx_p\dots\right]_{m_f} = \left[x_qx_px_q\dots\right]_{m_f}$ on a piece of the left hand side of length $m_f$. But since $k+l'+2 < 2m_f$, this only allows us to change either the first or the last letter but not both. So we necessarily have $k+l'+2=m_f$, so $ k = m_f-l'-2 = l-1$. 
   Applying the relation leads to $x_ix_j = x_px_q$ and $i=p$ and $j=q$. So $f=g=e$, $a = t^e_k$ and $b = t^e_{k+1}$ for some $k\in\{1,\dots,m_e-3\}$ and the corresponding relation is in $R$.

So we have indeed $abc\neq 1$, $abc\notin S$ and $abc^{-1}=e \Rightarrow abc^{-1}\in R$ for $a, b, c\in S$ so this is a restricted triangular presentation. 
\end{proof}
\begin{Bemerkung}
	Let $\Gamma$ be a simple graph, with edges labeled by numbers $\geq 2$ and with an orientation $o$ such that an edge is bioriented if and only if it has label $2$. Then it follows from the proof of Lemma \ref{dual presentation} that the dual presentation  of $A_{\Gamma}$ with orientation $o$ is always a presentation of $A_{\Gamma}$ but it is not always a restricted triangular presentation. For example if $\Gamma$ is a 3-cycle with one edge labeled by 2 and the two others labeled by $m, n \geq 3$, the presentation is not a restricted triangular presentation.  
\end{Bemerkung}
\begin{Bemerkung}
	In Corollary \ref{dihedral} we used the Garside structure on $DA_n$ induced by $G_{2,n}$. Lemma \ref{dual presentation} implies that $DA_n \cong G_{n,2}$, this presentation corresponds to another Garside structure. So in particular $G_{2,n}\cong G_{n,2}$ as groups but with different Garside structures.
\end{Bemerkung}

\begin{Theorem}\label{Artin}
Let $\Gamma$ be a simple graph, with edges labeled by numbers $\geq 2$ and with an orientation $o$ such that an edge is bioriented if and only if it has label $2$. Assume that 
every 3-cycle is directed and no 4-cycle is misdirected. Let $A$ be the Artin group associated to $\Gamma$. Then $A$ is Cayley systolic.
\end{Theorem}

\begin{proof}

 We prove using Theorem \ref{systolic} that $Flag(A_{\Gamma},S)$ with respect to the dual presentation with orientation $o$ is systolic. Note that with respect to the partial order defined in Remark \ref{order}:
 \begin{itemize}
  \item If $a, b\in S$ with $ ab\in S $ then $ab = \Delta_e$ for some $e\in E(\Gamma)$.
  \item For $s\in S$, $e\in E(\Gamma)$ we have $s\leq_L \Delta_e \Leftrightarrow s\leq_R \Delta_e$. 
  \item If $s\leq_L \Delta_e$ and $s\leq_L\Delta_f$ for some $s\in S$, $e,f\in E(\Gamma)$, $e\neq f$ \linebreak then $s\in \{x_1,x_2,\dots, x_n\}$. 
  \item If $s, t\leq_L \Delta_e$ and $s, t\leq_L\Delta_f$ for some $s, t\in S$, $s\neq t$ then $e=f$.
  \item If $x_i \leq_L \Delta_e$ then $x_i\in o(e)\cup i(e)$.
  \item If $\Delta_e =  x_ix_j$, $i(e)\cup o(e)=\{x_i, x_j\}$.
 \end{itemize}

 
 We check the different conditions of Theorem \ref{systolic}: 
 
 \ref{first condition} Assume $\exists u,w,a,b,c,d\in S$, $u\neq w$, $a\neq d$ with $ua=wb \in S$ and \linebreak $ud=wc\in S$. Then $ua=wb = \Delta_{e}$ and $ud=wc = \Delta_f$ for some $e, f\in E(\Gamma)$. As $u\neq w$ we have $e = f$. But then $a=d$ which is a contradiction. 
 
 \ref{second condition}Assume $\exists v,x,a,b,c, d\in S$, $v\neq x$, $a\neq b$ with $bv = cx\in S$ and $av = dx \in S$. Then $bv = cx = \Delta_e$ and $av=dx = \Delta_f$ for some $e,f\in E(\Gamma)$. As $v\neq x$ we have $e = f$. And so $a=b$ which is a contradiction.
 
 \ref{third condition} Assume $\exists u,v, x, b,c\in S$, $v\neq x$ with $ux, uv, vb, xc\in S$ and $vb=xc$. Then $uv = \Delta_e$, $ux = \Delta_f$ and $vb=xc = \Delta_g$ for some $e, f, g\in E(\Gamma)$. 
 \begin{itemize}                                                                                                                                                                                              \item If $e=f$ we have $x=v$ which is a contradiction.
  \item If $e=g$ and $e\neq f$ then $u,x \in\{x_1,\dots, x_n\}$. But then \linebreak $o(e)\cup i(e) = \{u,x\} = o(f)\cup i(f)$. So $e=f$, which is a contradiction.
  \item If $e\neq g$, $f\neq g$, $e\neq f$ then $u, v, x\in \{x_1, \dots, x_n\}$ and $u\in i(e)$, $u\in i(f)$, $v\in o(e)$, $x\in o(f)$, and $\{x,v\} = i(g)\cup o(g)$. But this corresponds to an undirected triangle in the defining graph $\Gamma$.                                                                                                                                                                                                                                                    \end{itemize}
 
 \ref{fourth condition} Assume $\exists v,w, x, a, b\in S$, $v\neq x$, with $vw, xw, dx,av\in S$ and $dx=av$. Then $vw = \Delta_e$, $xw = \Delta_f$ and $dx=av = \Delta_g$ for some $e, f, g\in E(\Gamma)$.
  \begin{itemize}                                                                                                                                                                                           \item If $e=f$ we have $x=v$ which is a contradiction.
  \item If $e=g$ and $e\neq f$ then $w,x \in\{x_1,\dots, x_n\}$. But then \linebreak $o(e)\cup i(e) = \{w,x\} = o(f)\cup i(f)$. So $e=f$, which is a contradiction. 
  \item If $e\neq g$, $f\neq g$, $e\neq f$ then $v,w, x\in \{x_1, \dots, x_n\}$ and $w\in o(e)$, $w\in o(f)$ and $v\in i(e)$, $x\in i(f)$ and $o(g)\cup i(g) = \{v,x\}$. But this corresponds to an undirected triangle in the defining graph $\Gamma$.                                                                                                                                                                                                                                                    \end{itemize}
 
 \ref{fifth condition} Assume $\exists u, v, w, x\in S$, $v\neq x, u\neq w$ with $wv, wx, uv, ux\in S$. Then $wv = \Delta_e$, $wx = \Delta_f$, $uv=\Delta_g$ and $ux = \Delta_h$ for some $e, f, g, h\in E(\Gamma)$. Then $v\neq x$ implies $e\neq f$ and $g\neq h$, and $u\neq w$ implies $e\neq g$ and $f\neq h$. So $u, v, w, x\in \{x_1,\dots,x_n\}$ and $v\in o(e)\cap o(g)$, $x\in o(f)\cap o(h)$, $w\in i(e)\cap i(f)$ and $u \in i(g)\cap i(h)$. Furthermore $i(e)\cup o(e) = \{w,v\}$ which 
  implies $v\neq w$ and $i(h)\cup o(h) = \{u, x\}$ which implies $u\neq x$. 
 So the 4-cycle $(u, v, w, x)$ is misdirected. 
 This is a contradiction to the orientation on $\Gamma$.

\end{proof}



\bibliographystyle{alpha}
\bibliography{MireilleBib}

\Addresses
\end{document}